\documentclass[12pt,reqno]{amsart}
\usepackage{latexsym,amsmath,amssymb,mathrsfs,setspace,pstricks,amsthm}
\pdfoutput=1
\newtheorem{theo}{Theorem}

\newtheorem{lemma}[theo]{Lemma}
\theoremstyle{remark}

\begin{document}
\doublespace

\title{ Units in $FD_{2p}$}

\author{ Kuldeep Kaur, Manju Khan}
\email{ kuldeepk@iitrpr.ac.in,manju@iitrpr.ac.in}

\address{Department of Mathematics, Indian Institute of Technology
Ropar,Nangal Road, Rupnagar - 140 001 }

\footnotetext{}
\keywords{Unitary units ; Unit Group; Group algebra.}

\subjclass[2000]{16U60, 20C05}

\maketitle

\markboth{ Kuldeep Kaur, Manju Khan}{Units in $FD_{2p}$}

\doublespacing
\begin{abstract}
In this paper, we present the structure of the group of
$*$-unitary units in the group algebra $FD_{2p}$, where $F$ is a finite field of
characteristic $p > 2$ , $D_{2p}$ is the dihedral group of order $2p$, and $*$ is the
canonical involution of the group algebra $FD_{2p}$. We also provide the
structure of the
maximal $p$-subgroup of the unit group $\mathscr{U}(FD_{2p})$ and compute a
basis of its
center.

\end{abstract}
\section*{introduction}
Let $FG$ be the group algebra of a group $G$ over a field $F$. For a normal
subgroup $H$ in $G$, the natural homomorphism $G \rightarrow G/H$ can be
extended to an algebra homomorphism from $FG$ to $F[G/H]$, defined by
$\displaystyle \sum_{g \in G} a_g g \mapsto \sum_{g \in G}a_g gH$. The kernel of
this homomorphism, denoted by $\Gamma(H)$, is the ideal generated by $\{ h -1 \
| \ h \in H \}$. Therefore,  $FG/ \Gamma(H) \cong F[G/H]$. In particular, for $H
= G$, $\Gamma(G)$ is known as the augmentation ideal of the group algebra $FG$.
Since $FG/\Gamma(G)\cong F$, it follows that the Jacobson radical $J(FG)$ is
contained in $\Gamma(G)$. The equality occurs if $G$ is a finite $p$-group and
$F$
is a field of characteristic $p$. In this case, $1 + \Gamma(G)$ is the same as
the
normalized unit group $V(FG)$ of the group algebra $FG$.   

For an ideal $I \subseteq J(FG)$, the natural homomorphism $FG \rightarrow FG/I$
induces an epimorphism from the unit group $\mathscr{U}(FG)$ to
$\mathscr{U}(FG/I)$, with kernel $1 + I$. Hence, $\mathscr{U}(FG)/ 1 + I \cong
\mathscr{U}(FG/ I)$.

If $\displaystyle x = \sum_{g \in G} x_gg $ is an element of $FG$, then the
element
$\displaystyle x^* = \sum_{g \in G} x_gg^{-1}$ is called the conjugate of
$x$. The map $x \mapsto x^*$ is an anti-automorphism of $FG$ of order 2, which
is
known as the canonical involution of the group algebra $FG$. 
An element  $x \in \mathscr{U}(FG)$ is called unitary if $x^*=x^{-1}$. The
unitary units of the unit group $\mathscr{U}(FG)$ form a subgroup
$\mathscr{U}_*(FG)$, and is called unitary subgroup of $\mathscr{U}(FG)$.

Let R be a ring. For $x, y \in R$, the operation $\circ$ in R can be defined by
$x \circ y =x+y+xy$. Clearly, $(R,\circ)$ is a monoid with 0 as the identity
element.
The elements invertible under
this operation are called quasi regular. Let $R^{\circ}$ be the set of all the
quasi regular elements of R. It is clear that $R^{\circ}$ is a group under the
operation $\circ$ and it is known as the
circle group of the ring R.  It is known that the map $r \rightarrow r-1$
defines a one-to-one correspondence
between subgroups of $\mathscr{U}(R)$ and subgroups of $R^{\circ}$.

If $G$ is a finite $p$-group and $F$ is a finite field of characteristic $p$,
then $\mathscr{U}(FG) = V(FG)\times F^*$, where $F^*$ is the cyclic group of all
nonzero elements of $F$. Sandling in\cite{MR761637} computed a basis
for the normalized unit group $V(FG)$, if $G$ is finite abelian $p$-group and
$F$ is a finite field with $p$ elements.  In
\cite{MR1136226}
Sandling provided generators and relators for each 2-group of order dividing
16 over finite field with 2 elements. Later on, Creedon and Gildea in
\cite{MR2884238} described the structure of $V(FD_8)$, where $F$ is a finite
field
of characteristic 2 and $D_8$ is the dihedral group of order $8$. \\
Sharma et.
al.\cite{FS3} described the structure of the unit group $\mathscr{U}(FD_{2p})$
of the dihedral group $D_{2p}$ for $p = 3$, over any finite
field $F$. For $p =5$, the structure of the unit group 
$\mathscr{U}(FD_{10})$ has been determined by Khan in
\cite{FD10}. Gildea in \cite{MR2534014} studied some properties of the center of
the maximal $p$-subgroup of the unit group  $\mathscr{U}(FD_{2p})$ over a finite
field $F$.
However, basis ot its center has not been yet determined.

The set all unitary units in the normalized unit group $V(FG)$ forms a subgroup
of $\mathscr{U}_*(FG)$. We
denote it by $V_*(FG)$. 
The unitary subgroup $\mathscr{U}_*(FG)$ coincides with $V_*(FG)$ if $F$ is a
finite field of
 characteristic 2. Otherwise, it coincides with $V_*(FG) \times \langle -1
\rangle$. Therefore, $V_*(FG)$ plays an important role in determining the
structure
of the unitary subgroup. Bovdi and Sakach in
\cite{FPG} described the structure of $V_*(FG)$, where $G$ is a finite abelian
$p$-group and $F$ is a finite field of characteristic
$p$. Bovdi and Erdei in \cite{F2D8} provided the structure of the
 unitary subgroup $V_*(F_2G)$, where $G$ is a nonabelian group of order 8 and
16.
 V. Bovdi and Rosa in \cite{order} computed the order of the unitary
subgroup of the
 group of units, when $G$ is either an extraspecial 2-group or the central
product of
 such a group with a cyclic group of order 4, and $F$ is a finite field of
characteristic 2 .They
 also computed the order of the unitary subgroup $V_*(F G)$, where $G$ is a
2-group
 with a finite abelian subgroup $A$ of index 2 and an element $b$ such that $b$
 inverts every element in $A$ and the order of $b$ is 2 or 4. If order of $b$ is
4, then V.
Bovdi and Rozgonyi in
 \cite{unitarybovdi} described the structure of $V_*(F_2G)$. Creedon and Gildea
in \cite{MR2512555} described the structure of the unitary
subgroup of the group algebra  $F_{2^k}Q_8$. 
Gildea \cite{MR2753766} provided the structure of the unitary subgroup of the
group algebra
$F_{2^k}D_8$. However, the structure of
$V_*(FD_{2p})$ is not known for a finite field of characteristic $p$.

Here, our goal is to study the structure and obtain 
generators of the unitary subgroup $\mathscr{U}_*(FD_{2p})$ for dihedral group
$D_{2p}$ of order $2p$ over a finite field $F$ of characteristic $p$. We also 
obtain a basis of the
center of the maximal p-subgroup of the unit group $\mathscr{U}(FD_{2p})$. 
Finally, we establish that the maximal $p$-subgroup of $\mathscr{U}(FG)$ is a
general product of the unitary subgroup with a metabelian group.
 
 Let $D_{2p} = \langle a,b \ | \ a^p=1= b^2, b^{-1}ab=a^{-1}\rangle$. Then, the
distinct conjugacy classes of $D_{2p}$ are $C_0=\{1\},\ C_i=\{a^i,a^{-i}\},
\mbox{for}\ 1 \leq i \leq l, \mbox{ where } l=\frac{p-1}{2} \mbox{ and } C =
\{b,ab,a^2b,\cdots , a^{p-1}b\}$. For a set $H$, where $\hat{H}$
denotes the sum of all the elements of $H$, $\{\hat{C_0},
\hat{C_1}, \hat{C_2}, \cdots , \hat{C_l}, \hat{C}\}$  forms a $F$-basis for the
center
$Z(FD_{2p})$ of the group algebra $FD_{2p}$.
                                                                                
 \section*{Unit group of $FD_{2p}$}

 \noindent If $A$ is the normal subgroup of $D_{2p}$ generated by element $a$
in $D_{2p}$, then it is clear that 
$FD_{2p}/\Gamma(A) \cong FC_2$. Since $\Gamma(A)$ is a nilpotent
ideal of
$FD_{2p}$, this implies that $\mathscr{U}(FD_{2p})/(1 + \Gamma(A)) \cong
\mathscr{U}(FC_2)$, which is isomorphic to $F^* \times F^* $. Let $\theta :
\mathscr{U}(FD_{2p}) \rightarrow \mathscr{U}(FC_2)$ be a group epimorphism
defined by \[ \sum_{i = 0}^{p -1}\alpha_ia_i + \sum_{j =
0}^{p-1}\beta_ja^jb \mapsto \sum_{i = 0}^{p -1}\alpha_i + \sum_{j =
0}^{p-1}\beta_j x,\] where
$C_2 = \langle x \rangle$. We can define a group homomorphism
$\psi:\mathcal{U}(FC_2) \rightarrow \mathcal{U}(FD_{2p})$ by 
$a_0+a_{1}x \rightarrow a_0 + a_{1}b$. Note that $\phi \circ \psi = 1$, and
therefore, 
$\mathscr{U}(FD_{2p}) \cong (1 + \Gamma(A))\rtimes F^* \times F^*
$. Now, $\Gamma(A)$ is a nilpotent ideal with index $p$. Hence, $1 +
\Gamma(A)$ is a nilpotent group with exponent $p$ and its nilpotency class is
at most $p-1$.
If order of $F$ is $ p^n$, then the order of $1 + \Gamma(A)$ is $ p^{4nl}$,
where $l
= \frac{p-1}{2}$. 

If $*$ is the canonical involution of the group algebra $FD_{2p}$, then we first
establish the structure of the unitary subgroup $\mathscr{U}_*(FD_{2p})$.

\section*{Structure of the Unitary Subgroup $\mathscr{U}_*(FD_{2p})$}
\begin{theo} If $A$ is the normal subgroup of $D_{2p}$, then the unitary
subgroup $\mathscr{U}_*(FD_{2p})$ of the group algebra $FD_{2p}$ is
the semidirect 
product of the normal subgroup $V_*(FA)$ with an elementary abelian 2-group,
where $V_*(FA)$ denotes the group of all unitary units in $V(FA)$.
\end{theo}
We need following lemmas: 
\begin{lemma}
The group of unitary units in $1 + \Gamma(A)$ is $V_*(FA)$, the unitary subgroup
of the normalized unit group $V(FA)$ of the group algebra $FA$. 
 \end{lemma}
 \begin{proof}
 Assume that $v=1+x_1+x_2b$, where $x_i \in \omega(FA)$ is an arbitrary element
of
$1 + \Gamma(A)$. Since $\omega(FA)$ is a nilpotent ideal, $1 + x_1$ is an
invertible element and thus $v$  can be written as $v=u(1+xb)$, where $u \in
V(FA)$ and $x \in \omega(FA)$. In particular,
if $v$ is a unitary unit, then from the equation $v^* v = 1$, we obtain
$u^*u+bx^*u^*uxb=1$
and
$bx^*u^*u+u^*uxb= 0$. Moreover, since $ by=y^*b$ for any $y \in FA$, we have $u 
\in V_*(FA)$. Further, note that  $u^{-1}v = 1+xb $ is a symmetric unit as well
as a unitary unit. Since the exponent of $ 1 + \Gamma(A)$ is $p$, it implies
that  $1
+ xb = 1 $ and hence $v = u$.
\end{proof}
 If $A = \langle a \rangle$ is a cyclic group of order $p$ and $F$ is a finite
field of order $p^n$,
then it is known that order of the group $V_*(FA)$ is $p^{\frac{n(p-1)}{2}}$ .
The following lemma provides a basis for $V_*(FA)$. Let $f(x)$ be a monic
irreducible polynomial of degree $n$
over $F_p$, such that $F \cong F_p[x]/ \langle f(x) \rangle$, and let $\alpha$
be
the residue class of $x \  \textrm{modulo} \ \langle f(x)\rangle$.
\begin{lemma}
 If $u_{i, k} = 1+\alpha^i(a-1)^k$,  then
the set \[\mathscr{B}=\{u_{i, k}^*u_{i, k}^{-1} \ | \ 0 \leq i \leq (n -1), 1
\leq k \leq (p-1) \textrm{and} \ k \ \textrm{is odd} \} \]  is a basis of
$V_*(FA)$. 
\end{lemma}
\begin{proof}
If $z_{i,k}=u_{i,k}^*u_{i,k}^{-1}$, then it is clear that $z_{i,k}$ is a
unitary unit.
To prove the lemma, first we establish the $p$-linear independence of the
elements of the given set $\mathscr{B}$. Assume that $z_{i_1,k_1},z_{i_2,k_2},
\cdots , z_{i_n,k_n}$ are distinct elements of $\mathscr{B}$ and
${z_{i_1,k_1}}^{s_1}{z_{i_2,k_2}}^{s_{2}} \cdots {z_{i_n,k_n}}^{s_n}=1$, such
that $ 1
\leq s_i \leq (p-1)$. Then, the element $u=
{u_{i_1,k_1}}^{s_1}{u_{i_2,k_2}}^{s_{2}}
\cdots {u_{i_n,k_n}}^{s_n}$ is a symmetric unit. Further, note that
$({u_{i_t,k_t}}^{s_t})^*= ((1+\alpha^{i_t}(a-1)^{k_t})^{s_t})^* \equiv 1- s_t
\alpha^{i_t}(a-1)^{k_t} \mod \ \omega^{k_{t}+1}(FA)$. If $k=min\{k_1, k_2,
\cdots , k_n \}$, then $(a-1)^{k_i} \in  \omega^{k+1}(FA)$ for $k_i > k$.
Therefore, $u^*=u \mod \
\omega^{k+1}(FA)$, and hence
$((1+\alpha^{i_1}(a-1)^{k_1})^{s_1})^*((1+\alpha^{i_2}(a-1)^{k_2})^{s_2})^*
\cdots ((1+\alpha^{i_n}(a-1)^{k_n})^{s_n})^* \equiv
((1+\alpha^{i_1}(a-1)^{k_1})^{s_1})((1+\alpha^{i_2}(a-1)^{k_2})^{s_2}) \cdots
((1+\alpha^{i_n}(a-1)^{k_n})^{s_n}) \mod \ \omega^{k+1}(FA)$. If 
$k_{j_1}=k_{j_2}= \cdots = k_{j_r}=k$, where $\{j_1, j_2, \cdots , j_r\}
\subseteq \{1,2, \cdots , n\}$, then $1-x = 1+x \mod \
\omega^{k+1}(FA)$, where $x = s_{j_1} \alpha^{j_1}(a-1)^k + s_{j_2}
\alpha^{j_2}(a-1)^k + \cdots + s_{j_r} \alpha^{j_r}(a-1)^k$. It follows that $x
\in \omega^{k+1}(FA)$, i.e., $\alpha(a-1)^k \in
\omega^{k+1}(FA)$, a contradiction, because  $(a-1)^k + \omega^{k+1}(FA)$ is a
basis of $\omega^{k}(FA) / \omega^{k+1}(FA)$.
\end{proof}
 Since $A$ is a normal subgroup of $D_{2p}$, $FD_{2p}/\Gamma(A)$ is an
algebra over $F$. Also note that $\Gamma(A)$ is $*$-stable nil ideal, and
therefore the set of all unitary units in $FD_{2p}/\Gamma(A)$ form a subgroup
of 
$\mathscr{U}(FD_{2p}/\Gamma(A))$ and is denoted by
$\mathscr{U}_*(FD_{2p}/\Gamma(A))$. 

\begin{lemma}
The unitary subgroup $\mathcal{U}_*(FD_{2p}/\Gamma(A))$ of
$\mathscr{U}(FD_{2p}/\Gamma(A))$
is the group 
generated by $\{-1 + \Gamma(A),b+ \Gamma(A) \}$.
\end{lemma}
 \begin{proof}
 If $u+\Gamma(A)$ is a unitary unit of $\mathcal{U}(FD_{2p}/\Gamma(A))$, then 
$ (u+\Gamma(A))^*(u+\Gamma(A)) = 1 +\Gamma(A)$ and hence
$u^*u$ is a symmetric unit in $1+\Gamma(A)$ because $\Gamma(A)$ is a $*$-stable
nil
ideal.
Further, if $u=x_0+x_1b$, where $x_0, x_1 \in FA$, then $uu^*=
x_0x_0^*+x_1x_1^*+2x_0x_1b$. For $x= \displaystyle \sum_{i=0}^{p-1}
(\alpha_ia^i+\beta_ia^ib)$, we can define $\chi(x)=\displaystyle
\sum_{i=0}^{p-1} (\alpha_i+\beta_i) $. Since
$uu^*$ is an element of $1 + \Gamma(A)$, it follows that
$\chi(x_0x_0^*+x_1x_1^*)=1$ and $\chi(x_0x_1)=0$. In particular,   
 if $x_0=\displaystyle \sum_{i=0}^{p-1}\alpha_ia^i$ and $x_1=\displaystyle
\sum_{i=0}^{p-1}\beta_ia^i$, then we obtain
\begin{equation}
 \sum_{i=0}^{p-1}\alpha_i^2+\displaystyle \sum_{i=0}^{p-1}\beta_i^2+ 2
\mathop{\sum_{i,j=0}^{p-1}}_{i<j} \alpha_{i}\alpha_{j}+
2\mathop{\sum_{i,j=0}^{p-1}}_{i<j} \beta_{i}\beta_{j}=1
\end{equation}
\begin{equation}
 (\displaystyle \sum_{i=0}^{p-1}
\alpha_i)(\displaystyle \sum_{i=0}^{p-1} \beta_j)=0. 
\end{equation}
 Now if $(\displaystyle \sum_{i=0}^{p-1} \alpha_i)=0$, then from equation (1) we
obtain
$(\displaystyle \sum_{i=0}^{p-1}\beta_i)^2=1$ and hence $\displaystyle
\sum_{i=0}^{p-1}\beta_i = \pm 1$. Thus, either $x_1$ or $-x_1$ is an element of
$V(FA)$. This implies that the unitary units in $FD_{2p}/\Gamma(A)$ are $\pm b +
\Gamma (A)$.
 Further, if $(\displaystyle \sum_{i=0}^{p-1} \beta_i)= 0$, then, in a similar
way, one
can show that the unitary units in $FD_{2p}/\Gamma(A)$ are $\pm 1 + \Gamma (A)$.
\end{proof}

\textbf{Proof of the theorem}:- Note that $V_*(FA)$ is a normal subgroup of 
$\mathcal{U}_*(FG)$ and $\langle V_*(FA), b, -1\rangle \subseteq
\mathcal{U}_*(FG)$.
Now suppose that $u \in \mathcal{U}(FG)\setminus V_*(FA)$ is a unitary unit,
then 
$u+\Gamma(A)$ is a unitary unit in $\mathcal{U}(FG/\Gamma(A)) $. If
$u+\Gamma(A)= b+\Gamma(A)$,
then $u=bx$ for some $x \in V_*(A)$. Hence,
$\mathcal{U}_*(FG)=V_*(FA) \rtimes
(\langle b \rangle
\times \langle  -1  \rangle)$.

\section*{The Structure of center $Z(1+\Gamma(A))$}

 To find a basis of center $Z(1+\Gamma(A))$ of $1 + \Gamma(A)$, first we
establish the structure of the
symmetric subgroup $S_*(FA)$ of the group algebra $FA$ in the following lemma:
\begin{lemma}
 If $u_{i,k}= 1+\alpha^i (a-1)^k$, then the set
\[\mathscr{B}= \{u_{i,k}^* u_{i,k} | 0 \leq i \leq (n-1), 1 \leq k \leq (p-1)
\textrm{and } k \textrm{ is even }\}\] forms a basis of $S_*(FA)$.
\end{lemma}
\begin{proof}
 Similar to the proof of the lemma(3)
\end{proof}
 
\begin{lemma}
 If $\omega_i=(a^i-a^{-i})(1 + b)$ and $\omega_i' = (a^i - a^{-i})(1 - b)$ for
$1 \leq i \leq l$, then the set $\{\omega_i, \omega_i', \omega_i\omega_i',
\omega_i'\omega_i \ | \  1\leq i \leq l\}$ is a free $F$- basis of $\Gamma(A)$
as a free $F$-module.  
\end{lemma}
\begin{proof}
It is known
that the set $\{ (a^i - 1), (a^i - 1)b \ | \ 1 \leq i \leq 2l\}$ is a free
$F$-basis of $\Gamma(A)$ as a free $F$-module. 
 Observe that $\omega_i\omega_j = 0, \omega_i'\omega_j' = 0$  for $1 \leq
i, j, \leq l$. Also note that $\omega_i \omega_i' = 2(a^{2i} + a^{-{2i}} - 2)(1
- b)$ and  $\omega_i' \omega_i = 2(a^{2i} + a^{-{2i}} - 2)(1+ b)$.\\
Thus, if $t=2i$, then 

\begin{equation*}
\left. \begin{aligned}
 (a^t-1)&= \frac{1}{4}(\omega_{2i}+\omega_{2i}')+\frac{1}{8}
(\omega_i\omega_i'+\omega_i'\omega_i)\\
(a^t-1)b&= \frac{1}{4}(\omega_{2i}-\omega_{2i}')-\frac{1}{8}
(\omega_i\omega_i'-\omega_i'\omega_i)
\end{aligned} 
\right\} 0 < 2i \leq l
\end{equation*}
\begin{equation*}
\left.\begin{aligned}
 (a^t-1)& =-\frac{1}{4}(\omega_{p-2i}+\omega_{p-2i}')+\frac{1}{8}
(\omega_i\omega_i'+\omega_i'\omega_i)\\
(a^t-1)b&=-\frac{1}{4}(\omega_{p-2i}-\omega_{p-2i}')-\frac{1}{8}
(\omega_i\omega_i'-\omega_i'\omega_i)
\end{aligned} 
\right\} l< 2i \leq(p-1)
\end{equation*}

If $t= p-2i$, then 
\begin{equation*}
\left.\begin{aligned}
(a^t-1)&=-\frac{1}{4}(\omega_{2i}+\omega_{2i}')+\frac{1}{8}
(\omega_i\omega_i'+\omega_i'\omega_i)\\
 (a^t-1)b&=-\frac{1}{4}(\omega_{2i}-\omega_{2i}')-\frac{1}{8}
(\omega_i\omega_i'-\omega_i'\omega_i)
\end{aligned}
 \right\} 0 < 2i \leq l
\end{equation*}
\begin{equation*}
\left.\begin{aligned}
(a^t-1)& =\frac{1}{4}(\omega_{p-2i}+\omega_{p-2i}')+\frac{1}{8}
(\omega_i\omega_i'+\omega_i'\omega_i)\\
 (a^t-1)b& =\frac{1}{4}(\omega_{p-2i}-\omega_{p-2i}')-\frac{1}{8}
(\omega_i\omega_i'-\omega_i'\omega_i)
\end{aligned} 
\right\} l < 2i \leq(p-1)
\end{equation*}
Therefore, $\Gamma(A)= \mbox{span} \{ \omega_i, \omega_i', \omega_i\omega_i',
\omega_i'\omega_i \ | \  1\leq i \leq l\}$. Since the dimension of $\Gamma(A)$
over $F$ is $4l$, the result follows.
\end{proof}
\begin{theo}
The center $Z(1 + \Gamma(A))$ is an
elementary abelian $p$ group of order $p^{n(l + 1)}$ with the set $\{ u_{i,k}^*
u_{i,k}, 1+ \alpha^i \hat{A}b \ | \ 0 \leq i \leq (n-1), 1 \leq k \leq (p-1)
\textrm{and } k \textrm{ is even }\}$ as a basis. 
\end{theo}
\begin{proof}
First we prove that $Z(1 + \Gamma(A)) = Z(FD_{2p}) \cap (1 + \Gamma(A))$.
Let $x \in Z(FD_{2p}) \cap (1 + \Gamma(A))$. It is clear that $x$ is of
the form 
\[x=1+ \sum_{i=1}^{l}\alpha_i(\hat{C_i} -2)+\beta\hat{A}b, \] where
$\alpha_i, \beta  \in F$. Since $Z(FD_{2p})\cap 1 + \Gamma(A)\subseteq Z(1 +
\Gamma(A))$, it implies that $|Z(1 + \Gamma(A))| \geq p^{n(l + 1)}$.
 To show the equality, we compute the dimension of $Z(\Gamma(A))$ over F. 
Take $x \in Z(\Gamma(A))$ such that \[x= \displaystyle \sum_{i=1}^{l}\alpha_i
\omega_i+\sum_{i=1}^{l}\alpha_i' \omega_i'+\sum_{i=1}^{l}\alpha_{ii}
\omega_{i}\omega_i'+ \sum_{i=1}^{l}\alpha_{ii}' \omega_i'\omega_i.\]  
 First, note that if $i = 2k + 1$, where $1 \leq i \leq l$ then
\[\omega_1\omega_i' = \omega_{k +
1}\omega_{k + 1}' - \omega_k\omega_k'\]  \[\omega_i'\omega_1 = \omega_{k +
1}'\omega_{k + 1} - \omega_k'\omega_k.\] For $i = 2k$, where $1 \leq i \leq l$
\[\omega_1\omega_i' = \omega_{l- k}\omega_{l - k}' - \omega_{l - (k
-1)}\omega_{l - (k -1)}' \] \[\omega_i'\omega_1 = \omega_{l- k}'\omega_{l - k} -
\omega_{l - (k -1)}'\omega_{l - (k -1)}. \]
 Next, observe that if $l$ is odd, then 
\[\omega_1\omega_j'\omega_j = 4 \omega_{2j+1}-4 \omega_{2j-1}-8 \omega_1, \
\mbox{
for}\ 1 \leq j \leq \frac{l-1}{2}\]
\[\omega_1\omega_j'\omega_j = -4 \omega_{2l-2j}+4 \omega_{2l-2j+2}-8 \omega_1,\
\mbox{ for}\  \frac{l+3}{2} \leq j \leq l-1 \]
\[\omega_1\omega_{\frac{l+1}{2}}'\omega_{\frac{l+1}{2}}=
-4\omega_{l-1}-4\omega_{l}-8\omega_{1}\]
\[\omega_1\omega_l'\omega_l= 4\omega_{2}-8\omega_{1}\] 
and if $l$ is even, then 
\[\omega_1\omega_j'\omega_j = 4 \omega_{2j+1}-4 \omega_{2j-1}-8 \omega_1, \
\mbox{
for} \ 1 \leq j \leq \frac{l-2}{2}\]
\[\omega_1\omega_j'\omega_j = -4 \omega_{2l-2j}+4 \omega_{2l-2j+2}-8 \omega_1,\
\mbox{ for} \ \frac{l+2}{2} \leq j \leq l-1 \]
\[\omega_1\omega_{\frac{l}{2}}'\omega_{\frac{l}{2}}=
-4\omega_{l}-4\omega_{l-1}-8\omega_{1}\]
\[\omega_1\omega_l'\omega_l= 4\omega_{2}-8\omega_{1}\]

   Now, after substituting these value in the equations $\omega_1 x=x \omega_1$
and $\omega_1' x=x \omega_1'$, we obtain
$\alpha_i=\alpha_i'= 0 \ \forall \ 1 \leq i \leq l$ and the following set of $l$
equations in $2l$
variables. 
  \[\alpha_{ii}'-\alpha_{{i + 1}{i + 1}}'-\alpha_{ii}+\alpha_{{i + 1}{i +
1}}=0\]
  \[-3(\alpha_{11}'- \alpha_{11})-2\displaystyle \sum_{i=2}^l(\alpha_{ii}'-
\alpha_{ii})= 0 \]
  Further, observe that the last equation can be written as  
\[-\displaystyle \sum_{k=1}^{l-1} (2k+1)(\alpha_{kk}'- \alpha_{k+1 k+1}'-
\alpha_{kk}+ \alpha_{k+1 k+1})= 0\] 
This implies that the dimension of the solution space of the above system of
linear equation is $ l+1$ and hence $|Z(1+\Gamma(A))| \leq p^{n(l+1)}$. 
Therefore,
$|Z(1+\Gamma(A))| = p^{n(l+1)}$. Since \[ Z(1+\Gamma(A))= \{1+
\sum_{i=1}^{l}\alpha_i(\hat{C_i} -2)+\beta\hat{A}b \ | \ \alpha_i, \beta  \in F
\}=
S_*(FA) \times (1+F \hat{A}b),\] the result follows from lemma (5).
\end{proof}

  \section*{Structure of $ 1+\Gamma(A)$}
 In this section, we obtain the structure of $1+\Gamma(A)$. We shall use the
following
result:

\begin{theo}[Pavesic \cite{PRM:6530508}]
 Let $(e_1, e_2, \cdots e_n)$ be an ordered $n$-tuple of orthogonal idempotents
in a
unital ring $R$ such that $1=e_1 + e_2 + \cdots + e_n$ and each $e_i$ strongly
preserves
a circle subgroup M of the circle group of R, i.e. $e_iM \subseteq M$ and $Me_1
\subseteq M$, then $(M, \circ)= (L, \circ) \circ (D,
\circ) \circ (U, \circ)$,
where $L= \{m \in M  | ( \forall \ i) \ e_i m = e_i m \overline{e_i}\}$,
$D = \{m \in M | ( \forall \ i) \ e_i m = e_i m {e_i}\}$ and $U = \{m \in
M | ( \forall \ i) \ e_i m = e_i m \underline{e_i}\}$, where
$\overline{e_i}= e_{i+1}+e_{i+2}+ \cdots + e_n$, $\underline{e_i}= e_{1}+e_{2}+
\cdots + e_{i-1}$, $\underline{e_1}=0$ and $\overline{e_n}=0$.                  
 \end{theo}

\begin{theo}
 If $A$ is the normal subgroup of $D_{2p}$, then $1 + \Gamma(A)$ is the general
product of unitary subgroup with a metabelian group.
\end{theo}

\begin{proof}
 Note that $\Gamma(A)$ is a circle group and $\{e_1= \frac{1 + b}{2}, e_2 =
\frac{1 - b}{2}\}$ is a complete set of orthogonal idempotents that strongly
preserve $\Gamma(A)$. Therefore, $(\Gamma(A),o) = (L, o )o(D, o)o (U, o)$, where
$(L, o) = \displaystyle \oplus_{i = 1}^l F\omega_i'$,
$(U, o) = \displaystyle \oplus_{i = 1}^l F\omega_i$ and $(D, o) = \displaystyle
\oplus_{i = 1}^l F\omega_i\omega_i' \displaystyle \oplus_{i = 1}^l
F\omega_i'\omega_i$. Hence, $ 1+ \Gamma(A)$ is a general product of elementary
abelain $p$-groups $(1 + L), (1 + D),$ and $ 1+ U$. Moreover, $1 + D$ normalize
$1 + L$; thus the group $W = (1 + L)\rtimes (1 +
D) $ is a metabelian group. 

If $C_{1 + \Gamma(A)}(a)$ is a centralizer of $a$ in $1 + \Gamma(a)$, then it is
clear that it is an elementary abelian $p$-group of order $p^{np}$. Since
$V_*(FA)$ and $Z(1 + \Gamma(A))$ are in $C_{1 + \Gamma(A)}(a)$ and their
intersection is identity, it follows that $C_{1 + \Gamma(A)}(a) = V_*(FA) \times
Z(1 + \Gamma(A))$. Also note that $W \cap C_{1 + \Gamma(A)}(a) = Z(1 +
\Gamma(A))$. Thus we obtain that $1 + \Gamma(A) = W C_{1 + \Gamma(A)}(a) $ is a general product of $W$ and $V_*(FA)$. 
\end{proof}

If $\alpha$ is the residue class of $x \mod \langle f(x)\rangle$, then the set $
\{1 + \alpha^i\omega_j'\ | \ 0 \leq i \leq n-1,  1 \leq j \leq l\}$ generates
$(1
+ L)$ and $\{1 + \alpha^i\omega_j\ | \ 0 \leq i \leq  n-1, 1 \leq j \leq l\}$
generates $ (1 + U)$. The following lemma provides the structure of $1+D$.
\begin{lemma}
 The set $\{1+\alpha^{i}(a-a^{-1})^{2k}(1-b), 1+\alpha^{i}(a-a^{-1})^{2k}(1+b) \
| \  1 \leq k \leq l \textrm{ and } 0 \leq i \leq (n-1)\}$ forms a basis of
$1+D$.
\end{lemma}
\begin{proof}
First, we show that\[ \displaystyle
\sum_{i=1}^{l} F\omega_{i}\omega_{i}'= span \{(a-a^{-1})^{2k}(1-b)\ | \ 1 \leq k
\leq l\}.\]
For this, we show \[\omega_k\omega_k' \in  span\{(a-a^{-1})^{2k}(1-b)\ | \ 1 \leq
k \leq l\}\] by induction over $k$. \\ The result is trivial for
$k=1$, as $\omega_1\omega_1'= 2(a-a^{-1})^2(1-b)$.\\ Assume the result for
$k-1$.
 Consider $\omega_k \omega_k'$. Notice that \[\omega_k \omega_k'=
2(a^k-a^{-k})^2(1-b)= 2(a^{2k}+a^{-2k}-2)(1-b). \] 
Now, 
\begin{eqnarray*}
(a-a^{-1})^{2k}&= &\displaystyle \sum_{j=0}^{k-1} (-1)^j{2k \choose
j}(a^{2k-2j}+a^{-(2k-2j)})+ (-1)^k {2k \choose k}\\
 &= &\displaystyle \sum_{j=0}^{k-1} (-1)^j{2k \choose
j}(a^{2k-2j}+a^{-(2k-2j)}-2)\\
(a-a^{-1})^{2k}(1-b)&=&\displaystyle \sum_{j=0}^{k-1} (-1)^j \frac{1}{2}{2k
\choose
j}\omega_{k-j}\omega_{k-j}'
\end{eqnarray*}
  Therefore, $\frac{1}{2}\omega_k\omega_{k}'=
(a-a^{-1})^{2k}(1-b)-(\displaystyle \sum_{j=1}^{k-1} (-1)^j \frac{1}{2}{2k
\choose
j}\omega_{k-j}\omega_{k-j}') $. Hence, by the induction hypothesis, we obtain  
\[\omega_k\omega_{k}'\in span\{(a-a^{-1})^{2k}(1-b)\ | \ 1 \leq k \leq l\}.\] 
Therefore,
\[\displaystyle
\sum_{i=1}^{l} F\omega_{i}\omega_{i}'=
\displaystyle \sum_{k=1}^{l}F(a-a^{-1})^{2k}(1-b).\]
 It implies that 
\begin{eqnarray*}
1+ \displaystyle
\sum_{i=1}^{l} F\omega_{i}\omega_{i}'&=&
1+ \displaystyle \sum_{k=1}^{l}F(a-a^{-1})^{2k}(1-b)\\
&=&\displaystyle \prod_{k=1}^{l} \prod_{i=0}^{n-1}
(1+\alpha^{i}(a-a^{-1})^{2k})(1-b)
\end{eqnarray*} 
  Similarly, we can show that
\[1+ \displaystyle
\sum_{i=1}^{l} F\omega_{i}'\omega_{i}=  \displaystyle \prod_{k=1}^{l}
\prod_{i=0}^{n-1}
(1+\alpha^{i}(a-a^{-1})^{2k})(1+b)\]
 Since $ 1+D = (1+ \displaystyle
\sum_{i=1}^{l} F\omega_{i}\omega_{i}')\times (1+ \displaystyle \sum_{i=1}^{l}
F\omega_{i}'\omega_{i})$, the result follows.
 \end{proof}
 \bibliographystyle{plain}
 \bibliography{references}

\end{document}